\documentclass[final]{siamltex}

\usepackage{tikz}
\usepackage{pgfplots}
\usepackage{verbatim}

\usepackage{graphicx}
\usepackage{amssymb}

\renewcommand{\grad}{\mathop{\rm grad}\nolimits}
\renewcommand{\div}{\mathop{\rm div}\nolimits}

\title{Splitting schemes for hyperbolic heat conduction equation}

\author{Petr N. Vabishchevich\thanks{Keldysh Institute of Applied Mathematics, 
  4 Miusskaya Square, 125047 Moscow, Russia ({\tt vabishchevich@gmail.com}).}
}

\begin{document}

\maketitle

\begin{abstract}
Rapid processes of heat transfer are not described by the standard heat conduction equation.  
To take into account a finite velocity of heat transfer, we use the hyperbolic model of heat conduction, 
which is connected with the relaxation of heat fluxes.  
In this case, the mathematical model is based on a hyperbolic equation of second order
or a system of equations for the temperature and heat fluxes.  
In this paper we construct for the hyperbolic heat conduction equation the additive schemes of splitting with respect to directions.  
Unconditional stability of locally one-dimensional splitting schemes is established.  
New splitting schemes are proposed and studied for a system of equations written in terms of the temperature and heat fluxes.
\end{abstract}



\pagestyle{myheadings}
\thispagestyle{plain}
\markboth{PETR N. VABISHCHEVICH}{SPLITTING SCHEMES FOR HYPERBOLIC HEAT}

\section{Introduction}

Linear parabolic theories of diffusion and heat conduction are based on the Fick and Fourier laws, respectively, 
and predict an infinite speed of propagation \cite{crank1979mathematics,incropera1996fundamentals}. 
In this case, the amplitude of propagating perturbations decreases exponentially 
with the distance and the infinite speed of perturbations can often be ignored.  
Nevertheless, in many applied problems the wave nature of heat transfer should be taken into account.  
Since paper \cite{cattaneo1958forme}, various corrections have been proposed for parabolic heat conduction models in order
to eliminate the paradox of infinite speed of perturbation propagation
 \cite{joseph1989heat,casas2010extended,shashkov2004wave}.  

The standard parabolic heat conduction model is based on the explicit representation of the heat flux through the temperature gradient.  
The hyperbolic heat conduction model includes an additional term with the time derivative for the heat flux 
which is proportional to the relaxation tensor.  
More general models (see \cite{joseph1989heat}) in addition includes the relaxation of the temperature gradient.  
Separate attention should be given to the hyperbolic model of convection-diffusion for moving media
\cite{gomez2008hyperbolic,gomez2008mathematical}.

Two possibilities can be highlighted in constructing computational methods for 
the approximate solution of hyperbolic heat transfer problems.  
The first is connected with the transition from a system of  the first order evolutionary 
equations for the temperature and heat flux to a single hyperbolic equation of second order.  
In contrast to the standard parabolic equation of heat conduction there does present a term with the second time derivative.  
The second possibility is based on the usage of the initial scalar-vector system of equations.  

To solve approximately the boundary value problem for the hyperbolic heat conduction equation,
classical numerical methods can be used including finite-difference approximations in space, 
finite volume schemes or finite-element approximations. For instance, three-level difference schemes for the hyperbolic 
heat conduction equation are constructed in  \cite{ciegis-numerical,samarskii1995computational}.
To investigate the stability and convergence of difference schemes, the general theory of stability for operator-difference schemes
is used in \cite{samarskii2001theory,samarskii2002difference}. 
Investigation of the stability on the basis of a priori estimates of the finite-difference solution for 
the model with the relaxation of the temperature gradients is given in  \cite{dai2004unconditionally,zhang2001unconditionally}.
An analysis of possibilities to use the simplest schemes of first and second order for integration 
in time is given in \cite{moosaie2009comparative} for model one-dimensional problems
of hyperbolic heat conduction.  

The system of equations governing thermal processes in terms of the temperature and heat flux 
has a defined structure with conjugated to each other operators.  
Such a structure of the mathematical model makes possible to use 
this feature in the construction of computational algorithms  
\cite{brezzi1991mixed,roberts1991mixed}.
In a number of papers (see, eg, 
\cite{glass1985numerical,shen2003numerical,tamma1991hyperbolic,yang1990characteristics}) 
the hyperbolic nature of mathematical models with heat waves emphasizes 
in using traditional technologies of compressible media dynamics.  
Various finite element methods are applied in papers \cite{gomez2007discontinuous,gomez2007finite} 
for equations of hyperbolic convection-diffusion theory.  
 
Much attention is paid to the construction of the additive schemes (splitting schemes) 
for approximate solving initial-boundary value problems for multi-dimensional partial differential equations 
\cite{marchuk1990splitting,yanenko1971method}.
Transition to a sequence of more simple problems allows to construct, for example, economical 
difference schemes - schemes based on the splitting with respect to spatial variables.
In some cases it is reasonable to perform splitting with respect to subproblems of different nature -
splitting in physical processes.
At present regionally-additive schemes (domain decomposition methods) 
are actively discussed \cite{samarskii1999additive}.  These schemes are oriented to the construction of 
computational algorithms for parallel computers.

Additive difference schemes in general conditions of the splitting of the problem operator into a sum of noncommutative 
non-selfadjoint operators are obtained in the most simple way for the case of two-component splitting.
In this case for the evolutionary equation of first order the classical alternative direction schemes, factorized 
and predictor-corrector schemes are unconditionally stable at weak restrictions.
A more complicated situation takes place in the case of multi-component splitting (splitting into three and more operators).
For these problems the most interesting results are obtained on the basis of the concept of summarized approximation.
The initial problem at the transition from one time level to another is divided into several subproblems, 
and each of these subproblems, in general, do not approximate the initial problem.  
On this way, unconditionally stable schemes of componentwise splitting (locally one-dimensional schemes 
of splitting with respect to spatial variables) are constructed.  

A new class of operator-difference splitting schemes - vector additive schemes - was developed in
papers \cite{abrashin1990variant,vabishchevich1996vector}.  
In this class of schemes we go from the initial scalar problem for one unknown function 
to the problem for a vector, each component of which can be treated as the solution of the problem.  
On this way we construct the full approximation schemes for evolutionary equations of the first 
and second order based on a general multi-component splitting.  
New additive difference schemes for differential-operator equations of the first and second order 
for the general case of splitting with an arbitrary number of pairwise noncommutative
operator terms were constructed in \cite{samarskii1998regularized,samarskii1999additive}
using the principle of regularization.  

The main theoretical results on the stability and convergence of the additive schemes 
were obtained for scalar evolutionary equations of the first order and, in some cases, for second-order equations.  
Splitting schemes for systems of evolutionary equations are of considerable interest for computational practice.  
For standard parabolic and hyperbolic systems of equations with selfadjoint elliptic operators 
additive schemes were constructed in \cite{samarskii2001theory}
using the principle of regularization for difference schemes.  
The Cauchy problem for a special linear system of first order equations in the Hilbert 
space with the conjugate operators (divergence and gradient) is considered in paper \cite{VabSys}.  
Such a structure of equations is characteristic for the considering here problems of hyperbolic heat transfer.  

In the present work there are constructed splitting schemes with respect to spatial variables for 
the approximate solving the equation of hyperbolic heat conduction.  
Unconditionally stable locally one-dimensional difference schemes are constructed here both
for a single heat conduction equation and for the system of equations based on the temperature and heat flux as unknowns. 
This paper is organized as follows.  
In section 2 the differential problem is formulated for the hyperbolic heat conduction. 
Appropriate a priori estimates are obtained for the solution of the hyperbolic equation in the both
above mentioned formulations.   
Approximation in space is discussed in Section 3 for a model problem in a rectangle.
It was shown that the grid operators of divergence and gradient are ajoint each other. 
Standard three-level difference schemes for the hyperbolic heat conduction equation 
are constructed in Section 4. The a priori estimates are derived for the difference solution.
Difference schemes for the system of equations based on the temperature and heat flux as unknowns are considered in Section 5.  
Unconditionally stable schemes are derived via the regularization of explicit-implicit schemes.  
Locally one-dimensional schemes for the hyperbolic heat conduction equation are studied in Section 6.  
Splitting scheme for the system of hyperbolic heat conduction equations are proposed in Section 7.  

\section{Differential problem}
  
Temperature $u(\mathbf{x},t)$  in bounded domain $\Omega$  with boundary  $\partial\Omega$
is governed by the equation  
\begin{equation}\label{2.1}
  c \frac{\partial u}{\partial t} + \div \mathbf{q} = f,
  \quad \mathbf{x} \in \Omega,
  \quad 0 < t \leq T,
\end{equation}
where $\mathbf{x} = (x_1, x_2, ..., x_n)$ is a point in space, 
$t$  is the time  ($T > 0$), and $\mathbf{q} = \mathbf{q}(\mathbf{x},t)$  is the heat flux.  
In (\ ref (2.1)) $c = c(\mathbf{x}) \geq c_0 > 0$ is the specific heat capacity of a medium, 
and  $f = f(\mathbf{x},t)$ is the rate of  volumetric heat sources.  
The standard (parabolic) model of the heat conduction results from the following 
representation for the heat flux (Fourier's law)  
\begin{equation}\label{2.2}
  \mathbf{q} + k \grad u = 0,
\end{equation}
where $k = k(\mathbf{x}) \geq k_0 > 0$  is the thermal conductivity of the medium.  
Substitution of (\ref{2.2})  in (\ref{2.1}) leads us to the parabolic heat conduction equation 
\begin{equation}\label{2.3}
  c \frac{\partial u}{\partial t} - 
  \div( k \grad u ) = f,
  \quad \mathbf{x} \in \Omega,
  \quad 0 < t \leq T,
\end{equation}
supplemented by appropriate boundary and initial conditions.

In the model of the hyperbolic heat conduction instead of  (\ref{2.2}) we use the following relation  
\begin{equation}\label{2.4}
  \mathbf{q} + \nu  \frac{\partial \mathbf{q}}{\partial t} + 
  k \grad u = 0,
\end{equation}
where  $\nu$  is the relaxation parameter for the heat flux.  
From (\ref{2.1}) and (\ref{2.4}) we obtain the hyperbolic heat conduction equation  
\begin{equation}\label{2.5}
  \nu c \frac{\partial^2 u}{\partial t^2} + 
  c \frac{\partial u}{\partial t} - 
  \div( k \grad u ) =
  f + \nu \frac{\partial f}{\partial t},
  \quad \mathbf{x} \in \Omega,
  \quad 0 < t \leq T .
\end{equation}

Consider a model boundary value problem for equation (\ref{2.5})  (system (\ref{2.1}) and (\ref{2.4})),
where the boundary conditions are  as follows
\begin{equation}\label{2.6}
  u(\mathbf{x},t) = 0,
  \quad \mathbf{x} \in \partial \Omega,
  \quad 0 < t \leq T .
\end{equation}
In addition, two initial conditions  are prescribed
\begin{equation}\label{2.7}
  u(\mathbf{x},0) = v_0(\mathbf{x}),
  \quad \frac{\partial u}{\partial t} (\mathbf{x},0) = v_1(\mathbf{x}),
  \quad \mathbf{x} \in \Omega.  
\end{equation}
The simplest a priori estimates for problem (\ref{2.5})--(\ref{2.7}), 
((\ref{2.1}), (\ref{2.4}), (\ref{2.6}), (\ref{2.7})) will be derived now in order
to be our guidelines in the investigation of grid problems.  

Let $(\cdot, \cdot)$  and $\| \cdot \|$ be the scalar product and norm in $\mathcal{H} = L_2(\Omega)$, respectively.  
Multiplying scalarly equation (\ref{2.5}) by $\partial u / \partial t$ in $\mathcal{H}$
we obtain  
\begin{equation}\label{2.8}
  \left ( c \frac{\partial u}{\partial t}, \frac{\partial u}{\partial t}\right ) +
  \frac{\nu}{2} \frac{d}{d t}
  \left ( c \frac{\partial u}{\partial t}, \frac{\partial u}{\partial t}\right ) +
  \frac{1}{2} \frac{d}{d t} (k \grad u, \grad u) =
  \left ( f + \nu \frac{\partial f}{\partial t}, \frac{\partial u}{\partial t}\right ) .   
\end{equation}
The right hand side of  (\ref{2.8}) is estimated as follows  
\begin{equation}\label{2.9}
  \left ( f + \nu \frac{\partial f}{\partial t}, \frac{\partial u}{\partial t}\right ) \leq   
  \left ( c \frac{\partial u}{\partial t}, \frac{\partial u}{\partial t}\right ) +
  \frac{1}{4}\left (c^{-1} \left (f + \nu \frac{\partial f}{\partial t} \right ), f + \nu \frac{\partial f}{\partial t} \right ) .
\end{equation}
From (\ref{2.8}), (\ref{2.9}) we have the inequality  
\begin{equation}\label{2.10}
   \frac{d}{d t} S \leq  \frac{1}{2} \left (c^{-1} f + \nu \frac{\partial f}{\partial t}, f + \nu \frac{\partial f}{\partial t} \right ) .
\end{equation}
Here  
\begin{equation}\label{2.11}
  S(t) = \nu 
  \left ( c \frac{\partial u}{\partial t}, \frac{\partial u}{\partial t}\right ) +
  (k \grad u, \grad u) 
\end{equation}
defines the squared norm for the solution of (\ref{2.5})--(\ref{2.7}) with boundary conditions (\ref{2.6}).
Applying to (\ref{2.10}) the Gronwall lemma, we obtain the desired estimate  
\begin{equation}\label{2.12}
  S(t) \leq  S(0) + 
  \frac{1}{2} \int\limits_0^t 
  \left \|c^{-1/2} \left (f + \nu \frac{\partial f}{\partial t} \right )(\mathbf{x},\theta) \right \|^2 d \theta .
\end{equation}
At $\nu =0$  estimate (\ref{2.12}) degenerates into the corresponding estimate for the solution of parabolic heat equation (\ref{2.3}).

For the system of equations instead of initial conditions (\ref{2.7}) it is more natural to use  
\begin{equation}\label{2.13}
  u(\mathbf{x},0) = v_0(\mathbf{x}),
  \quad \mathbf{q} (\mathbf{x},0) = \mathbf{g}_0(\mathbf{x}),
  \quad \mathbf{x} \in \partial \Omega,  
\end{equation}
ie instead of the rate of temperature variation we define  the heat flux.  
The transition from one to another initial conditions is provided by equation (\ref{2.3}).

To obtain a simple a priori estimate for  system (\ref{2.1}), (\ref{2.4})  we scalarly multiply 
equation (\ref{2.1})  by $u$,  and (\ref{2.4}) - by  $k^{-1} \mathbf{q}$ and sum them.  
This gives  
\begin{equation}\label{2.14}
   \frac{1}{2} \frac{d}{d t} (c u,u) +
   \frac{\nu}{2} \frac{d}{d t} ( k^{-1} \mathbf{q},\mathbf{q}) +
   ( k^{-1} \mathbf{q},\mathbf{q}) = (f,u) . 
\end{equation}
For the right hand side we use the estimate  
\[
  (f,u) \leq \frac{1}{2} (c u,u) + \frac{1}{2} (c^{-1} f,f) .
\]
From (\ref{2.14}) we obtain  
\begin{equation}\label{2.15}
   \frac{d}{d t} G \leq  G + (c^{-1} f,f) ,
\end{equation}
\begin{equation}\label{2.16}
  G(t) = (c u,u) + \nu ( k^{-1} \mathbf{q},\mathbf{q}).
\end{equation}
From (\ref{2.15}) we derive the estimate  
\begin{equation}\label{2.17}
  G(t) \leq \exp(t) G(0) + \int\limits_0^t \exp(t-\theta)
  \|c^{-1/2} f(\mathbf{x},\theta) \|^2 d \theta ,
\end{equation}
which ensures the stability of the solution of system (\ref{2.1}), (\ref{2.4}) with respect to initial data (\ref{2.13})  
and the right hand side.  

\section{Approximation in space}

Let us consider the 2D model problem of the hyperbolic heat conduction in the rectangle  
\[
  \Omega = \{ \ \mathbf{x} \ | \ \mathbf{x} = (x_1, x_2),
  \quad 0 < x_{\alpha} < l_{\alpha}, \quad \alpha = 1,2 \} .
\] 
Let  $q_{\alpha}, \ \alpha = 1,2$ be the Cartesian components of heat flux $\mathbf{q} = (q_1,q_2)$.
The system of equations (\ref{2.1}), (\ref{2.4})  in the coordinate-wise representation 
takes the form 
\begin{equation}\label{3.1}
  c \frac{\partial u}{\partial t} + 
  \sum_{\alpha =1}^{2} \frac{\partial q_{\alpha}}{\partial x_{\alpha}} = f,
\end{equation}
\begin{equation}\label{3.2}
  q_{\alpha} + \nu  \frac{\partial q_{\alpha}}{\partial t} + 
  k \frac{\partial u}{\partial x_{\alpha}} = 0,
  \quad \alpha = 1,2 .
\end{equation}

On the set of functions $u$, satisfying homogeneous boundary conditions (\ref{2.6}),
we define the operators  
\begin{equation}\label{3.3}
  \mathcal{A}_{\alpha} u = \frac{\partial u}{\partial x_{\alpha}},
  \quad \alpha = 1,2 .
\end{equation}
Taking into account that  
\[
  \int\limits_{\Omega} \frac{\partial u}{\partial x_{\alpha}} v d \mathbf{x} =
  - \int\limits_{\Omega} u \frac{\partial v}{\partial x_{\alpha}} d \mathbf{x} ,
\]
we have  
\begin{equation}\label{3.4}
  \mathcal{A}^*_{\alpha} v = - \frac{\partial v}{\partial x_{\alpha}},
  \quad \alpha = 1,2 
\end{equation}
for the conjugate operators.
In view of (\ref{3.3}), (\ref{3.4})  the system of equations (\ref{3.1}), (\ref{3.2})
with boundary conditions (\ref{2.6}) can be written in the following operator form  
\begin{equation}\label{3.5}
  q_{\alpha} + \nu  \frac{d q_{\alpha}}{d t} + 
  k \mathcal{A}_{\alpha} u = 0,
  \quad \alpha = 1,2 ,
\end{equation}
\begin{equation}\label{3.6}
  c \frac{d u}{d t} - 
  \sum_{\alpha =1}^{2} \mathcal{A}^*_{\alpha} q_{\alpha} = f .
\end{equation}
Thus, the system of equations governing the hyperbolic heat conduction does have the operator structure with conjugate operators.

For hyperbolic heat conduction equation (\ref{2.5})  the corresponding operator-differential 
equation has the form  
\begin{equation}\label{3.7}
  \nu c \frac{d^2 u}{d t^2} + 
  c \frac{d u}{d t} + \mathcal{D} u = 
  f + \nu \frac{\partial f}{\partial t} ,
\end{equation}
\begin{equation}\label{3.8}
  \mathcal{D} = \sum_{\alpha =1}^{2} \mathcal{D}_{\alpha},
  \quad \mathcal{D}_{\alpha} u = 
  \mathcal{A}^*_{\alpha} k \mathcal{A}_{\alpha},
  \quad \alpha = 1,2 .  
\end{equation}
Operator $\mathcal{D}$,  as well as its individual terms $\mathcal{D}_{\alpha}, \ \alpha = 1,2$, 
is selfadjoint and positive definite in $L_2(\Omega)$ on the set of functions satisfying 
boundary conditions (\ref{2.3}).

We want to preserve the above operator structure of the differential model for the hyperbolic heat conduction
after its approximation in space.  
For simplicity, we will consider the simplest difference approximations on uniform grids.  
In the considering problems it is natural to use for the scalar and vector unknowns staggered grids, 
where scalar variables and vector components employ their own grids.  
Such a technology is standard for problems of computational fluid dynamics \cite{versteeg2007introduction} 
and electrodynamics \cite{taflove2000computational}.

The temperature is defined at the nodes of a uniform rectangular grid in  $\Omega$:
\[
   \bar{\omega} = \{ \mathbf{x} \ | \ \mathbf{x} = (x_1, x_2),
   \quad x_\alpha = i_\alpha h_\alpha,
   \quad i_\alpha = 0,1,...,N_\alpha,
   \quad N_\alpha h_\alpha = l_\alpha, \quad \alpha = 1,2\} 
\]
and let $\omega$ be a set of internal nodes ($\bar{\omega} = \omega \cup \partial \omega$).  
The components of vector quantities are referred to the corresponding edges of the grid.  
We define  
\[
   \bar{\omega}_1 = \{ \mathbf{x} \ |  
   \ x_1 = (i_1 + 0.5) h_1,
   \ i_1 = 0,1,...,N_1-1,
   \ x_2 = i_2 h_2,
   \ i_2 = 0,1,...,N_2 \} ,
\]
\[
   \bar{\omega}_2 = \{ \mathbf{x} \ | 
   \ x_1 = i_1 h_1,
   \ i_1 = 0,1,...,N_1,
   \ x_2 = (i_2 + 0.5) h_2,
   \ i_2 = 0,1,...,N_2-1 \} 
\]
and $\bar{\omega}_{\alpha} = \omega_{\alpha} \cup \partial \omega_{\alpha}, \ \alpha = 1,2$. 
Component of heat flux $q_{\alpha}, \ \alpha = 1,2$ will be evaluated on the
grid $\bar{\omega}_{\alpha}, \ \alpha = 1,2$ (Fig.\ref{f-1}).

\begin{figure}[ht] 
  \begin{center}
	\begin{tikzpicture}[domain=0:4]
	\draw[thin,color=gray] (0,0) grid (8,8);
	\foreach \x in {0,...,4}
		\foreach \y in {0,...,4}
		 	\filldraw [black] (2*\x,2*\y) circle (0.1);
	\foreach \x in {0,...,4}
		\foreach \y in {0,...,3}
		 	\draw [black] (2*\x,2*\y+1) circle (0.1);
	\foreach \x in {0,...,3}
		\foreach \y in {0,...,4}
		 	\draw [black] (2*\x+0.9,2*\y-0.1) rectangle +(0.2,0.2);
	\end{tikzpicture}
    \caption{Ñåòêè: {\large $\bullet$} --- $\bar{\omega} (u)$, $\square$ --- $\bar{\omega}_1 (q_1)$, 
       {\large$\circ$} --- $\bar{\omega}_2 (q_2)$.} 
	\label{f-1}
  \end{center}
\end{figure}
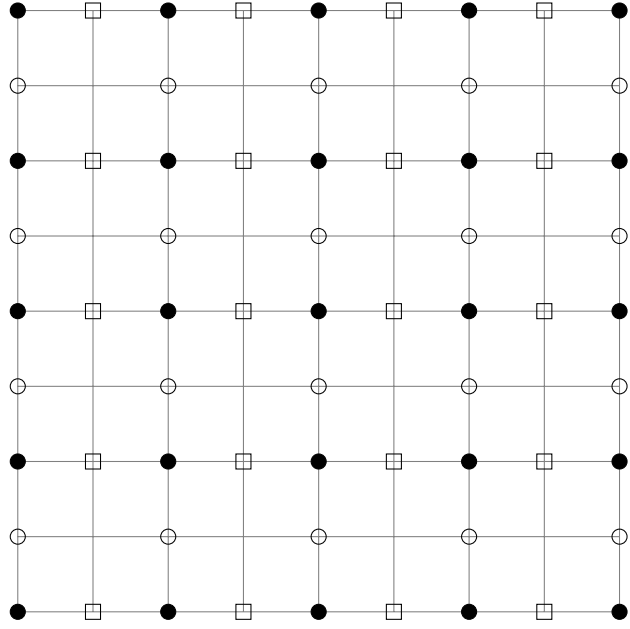

For grid functions  $y(\mathbf{x}) = 0, \ \mathbf{x} \in \partial \omega$
we define the Hilbert space $H = L_2({\omega})$ 
with the scalar product and norm  
\[
  (y,w) \equiv \sum_{{\bf x} \in \omega}
  y({\bf x}) w({\bf x}) h_1 h_2,
  \quad \|y\| \equiv (y,y)^{1/2} .
\]
Similarly, for the grid functions defined on grid $\omega_{\alpha}, \ \alpha = 1,2$, 
we define the Hilbert space $H_{\alpha}, \ \alpha = 1,2$,  where  
\[
  (y,w)_{\alpha} \equiv \sum_{{\bf x} \in \omega_{\alpha}}
  y({\bf x}) w({\bf x}) h_1 h_2,
  \quad \|y\|_{\alpha} \equiv (y,y)_{\alpha}^{1/2}, 
  \quad  \alpha = 1,2.
\]

Let us construct the grid analogs of differential operators 
$\mathcal{A}_{\alpha}, \ \mathcal{A}^*_{\alpha}, \ \alpha = 1,2$,
defined above according to  (\ref{3.3}), (\ref{3.4}).
We will use the standard  \cite{samarskii2001theory} 
 central-difference approximations for derivatives in space.  
We set  
\begin{equation}\label{3.9}
  (A_1 y)(\mathbf{x}) = \frac{y(x_1+0.5h_1, x_2)-y(x_1-0.5h_1, x_2)}{h_1},
  \quad \mathbf{x} \in \omega_1 ,
\end{equation}
so that $A_1: H \rightarrow H_1$.
Similarly, we define $A_2: H \rightarrow  H_2$, where  
\begin{equation}\label{3.10}
  (A_2 y)(\mathbf{x}) = \frac{y(x_1, x_2+0.5h_2)-y(x_1, x_2-0.5h_2)}{h_2},
  \quad \mathbf{x} \in \omega_2 .
\end{equation}
By the construction  we have
\begin{equation}\label{3.11}
  A_{\alpha} u = 
  \mathcal{A}_{\alpha} u + O(h_{\alpha}^2),
  \quad  \alpha = 1,2 .
\end{equation}

Direct calculations verify that for the adjoint operators 
$A^*_{\alpha}: H_{\alpha} \rightarrow H, \ \alpha = 1,2$ we  have the representation  
\begin{equation}\label{3.12}
  (A^*_1 y)(\mathbf{x}) = - \frac{y(x_1+0.5h_1, x_2)-y(x_1-0.5h_1, x_2)}{h_1},
  \quad \mathbf{x} \in \omega ,
\end{equation}
\begin{equation}\label{3.13}
  (A^*_2 y)(\mathbf{x}) = - \frac{y(x_1, x_2+0.5h_2)-y(x_1, x_2-0.5h_2)}{h_2},
  \quad \mathbf{x} \in \omega .
\end{equation}
For sufficiently smooth functions  $u$
\begin{equation}\label{3.14}
  A^*_{\alpha} u = 
  \mathcal{A}^*_{\alpha} u + O(h_{\alpha}^2),
  \quad  \alpha = 1,2 
\end{equation}
holds.

After approximation of system (\ref{3.5}), (\ref{3.6})  in space  we obtain the system of evolutionary equations  
\begin{equation}\label{3.15}
  q_{\alpha} + \nu  \frac{d q_{\alpha}}{d t} + 
  k A_{\alpha} u = 0,
  \quad \mathbf{x} \in \omega_{\alpha} ,
  \quad \alpha = 1,2 , 
\end{equation}
\begin{equation}\label{3.16}
  c \frac{d u}{d t} - 
  \sum_{\alpha =1}^{2} A^*_{\alpha} q_{\alpha} = f ,
  \quad \mathbf{x} \in \omega .
\end{equation}
Similarly,  equation (\ ref (3.7)) is associated with the evolutionary equation 
\begin{equation}\label{3.17}
  \nu c \frac{d^2 u}{d t^2} + 
  c \frac{d u}{d t} + D u = 
  f + \nu \frac{\partial f}{\partial t} ,
  \quad \mathbf{x} \in \omega ,
\end{equation}
\begin{equation}\label{3.18}
  D = \sum_{\alpha =1}^{2} D_{\alpha},
  \quad D_{\alpha}  = A^*_{\alpha} k A_{\alpha},
  \quad \alpha = 1,2 .  
\end{equation}
Taking into account (\ref{3.9}),(\ref{3.10}) and (\ref{3.12}),(\ref{3.13}), 
for grid operators $D_{\alpha}: H \rightarrow H,$ $\alpha = 1,2$
we obtain  
\begin{equation}\label{3.19}
  (D_1 y)(\mathbf{x}) = \frac{1}{h_1} \left (
  k(x_1+0.5h_1, x_2) \frac{y(x_1+h_1, x_2)-y(\mathbf{x})}{h_1}
  \right .
\end{equation}
\[
  \left .
  - k(x_1-0.5h_1, x_2) \frac{y(\mathbf{x})- y(x_1-h_1, x_2)}{h_1}
  \right ), 
  \quad \mathbf{x} \in \omega ,
\]
\begin{equation}\label{3.20}
  (D_2 y)(\mathbf{x}) = \frac{1}{h_2} \left (
  k(x_1, x_2+0.5h_2) \frac{y(x_1, x_2+h_2)-y(\mathbf{x})}{h_2}
  \right .
\end{equation}
\[
  \left .
  - k(x_1, x_2-0.5h_2) \frac{y(\mathbf{x})- y(x_1, x_2-h_2)}{h_2}
  \right ), 
  \quad \mathbf{x} \in \omega .
\]
Similarly  to (\ref{3.11}),(\ref{3.14}), we have  \cite{samarskii2001theory,samarskii1989numerical}
\begin{equation}\label{3.21}
  D_{\alpha} u = 
  \mathcal{D}_{\alpha} u + O(h_{\alpha}^2),
  \quad  \alpha = 1,2 
\end{equation}
in the class of sufficiently smooth coefficients $k$ and functions  $u$.
In addition, in the space of grid functions  $H$
\begin{equation}\label{3.22}
  D_{\alpha} = D^*_{\alpha},
  \quad k_0 \delta_{\alpha} E \leq D_{\alpha} \leq k_1 \Delta_{\alpha} E,
\end{equation}
\[
  \delta_{\alpha} = 
  \frac{4}{h^2_{\alpha}} \sin^2 \frac{\pi h_{\alpha}}{2 l_{\alpha}} ,
  \quad \Delta_{\alpha} = 
  \frac{4}{h^2_{\alpha}} \cos^2 \frac{\pi h_{\alpha}}{2 l_{\alpha}} ,
  \quad  \alpha = 1,2 ,
\]
where $ E $ is the unit (identity) operator and  
$k(\mathbf{x}) \leq  k_1, \ \mathbf{x} \in \Omega$.

\section{Difference schemes for the hyperbolic heat conduction equation}

We consider approximation in  time for the approximate solution of differential-operator equation (\ref{3.17}),
which is supplemented by the initial conditions  
\begin{equation}\label{4.1}
  u(\mathbf{x},0) = v_0(\mathbf{x}),
  \quad \frac{d u}{d t} (\mathbf{x},0) = v_1(\mathbf{x}),
  \quad \mathbf{x} \in \omega.  
\end{equation}
Let us define a uniform grid in time  
\[
  \overline{\omega}_\tau =
  \omega_\tau\cup \{T\} =
  \{t_n=n\tau,
  \quad n=0,1,...,N,
  \quad \tau N=T\} 
\]
and denote $y^n = y(t_n), \ t_n = n \tau$.
Standard three-level difference schemes with the second-order approximation in time will be considered.  

Equation (\ref{3.7})  is approximated by the difference scheme with weights  
\begin{equation}\label{4.2}
  \nu c \frac{u^{n+1} - 2u^{n} + u^{n-1}}{\tau^2} + 
  c \frac{u^{n+1} - u^{n-1}}{2\tau} 
\end{equation}
\[
  + D (\sigma u^{n+1} + (1- 2\sigma) u^{n} + \sigma u^{n-1}) = 
  \varphi^n , 
  \quad n = 1,2, \ldots, N-1 , 
\]
where  
\[
  \varphi^n = f^n + \nu \frac{f^{n+1} - f^{n-1}}{2\tau}, 
\]
with the corresponding initial data  
\begin{equation}\label{4.3}
  u^0=  v_0, 
  \quad  \frac{u^{1} - u^{0}}{\tau} = w_0 . 
\end{equation}

Scheme (\ref{4.2}), (\ref{4.3}) which belongs to the class of three-level operator-difference schemes 
can be investigated on the basis of the Samarskii stability (correctness) theory of operator-difference schemes.
Coincident  necessary and sufficient conditions of the stability of these schemes in various norms are obtained in \cite{samarskii2001theory,samarskii2002difference}.
With this in mind, we give here only the simplest a priori estimates of stability with respect to 
the initial data and right hand side for scheme  (\ref{4.2}), (\ref{4.3}).

\begin{theorem} 
\label{t-1} 
Difference scheme  (\ref{4.2}), (\ref{4.3})
 is unconditionally  stable at $\sigma \geq 0.25$ and for the finite-difference solution we have the estimate  
\begin{equation}\label{4.4}
  S^{n+1} \leq S^n + \frac{\tau }{2} (c^{-1} \varphi^n,\varphi^n) ,
\end{equation}
where  
\begin{equation}\label{4.5}
  S^n = \left (\left (\nu c E + \left (\sigma - \frac{1}{4} \right ) 
  \tau^2 D \right )\frac{u^n - u^{n-1}}{\tau }, 
  \frac{u^n - u^{n-1}}{\tau } \right ) 
\end{equation}
\[
  + \left (D \frac{u^n + u^{n-1}}{2}, \frac{u^n + u^{n-1}}{2} \right ) .
\]
\end{theorem} 
 
\begin{proof}
To prove this, we introduce the notation  
\[
  \zeta^n = \frac{u^n + u^{n-1}}{2} ,
  \quad \eta^n = \frac{u^n - u^{n-1}}{\tau } .
\]
Taking into account the identities  
\[
  u^n = \frac{1}{4} ( u^{n+1} + 2 u^n + u^{n-1})
  - \frac{1}{4} ( u^{n+1} - 2 u^n + u^{n-1}),
\]
\[
  \sigma u^{n+1} + (1- 2\sigma) u^{n} + \sigma u^{n-1} = 
  u^n  + \sigma ( u^{n+1} - 2 u^n + u^{n-1})
\]
we rewrite (\ref{4.2}) in the form 
\begin{equation}\label{4.6}
  \left (\nu c E + \left (\sigma - \frac{1}{4} \right ) \tau^2 D \right ) \frac{\eta^{n+1} - \eta^{n}}{\tau} + 
  c \frac{\eta^{n+1} + \eta^{n}}{2} + 
  D \frac{\zeta^{n+1} + \zeta^{n}}{2} = \varphi^n .
\end{equation}
We scalarly  multiply in $H$ this equation  by  
\[
  2(\zeta^{n+1} - \zeta^{n}) = \tau (\eta^{n+1} + \eta^{n}) .
\]
This gives    
\begin{equation}\label{4.7}
  \left (\left (\nu c E + \left (\sigma - \frac{1}{4} \right ) \tau^2 D \right )\eta^{n+1}, \eta^{n+1}\right ) +
  (D \zeta^{n+1}, \zeta^{n+1})  
\end{equation}
\[
  - \left (\left (\nu c E + \left (\sigma - \frac{1}{4} \right ) \tau^2 D \right )\eta^{n}, \eta^{n}\right ) -
  (D \zeta^{n}, \zeta^{n}) 
\]
\[
  + \frac{\tau }{2} (c (\eta^{n+1} + \eta^{n}), (\eta^{n+1} + \eta^{n})) = 
  \tau (\varphi^n,(\eta^{n+1} + \eta^{n})) .
\]
If $\sigma \geq 0.25$ then value  
\[
  S^n = \left (\left (\nu c E + \left (\sigma - \frac{1}{4} \right ) \tau^2 D \right )\eta^{n}, \eta^{n}\right ) +
  (D \zeta^{n}, \zeta^{n}) 
\]
defines the squared norm of the difference solution.  
With this notation we obtain the required estimate  (\ref{4.5}).
\qquad\end{proof} 
 
Estimate (\ref{4.5}) for the numerical solution is consistent with estimate  (\ref{2.12}) 
for the solution of the differential problem. 
Using this estimate it is easy to prove in the standard enough way \cite{samarskii2001theory,samarskii2002difference}
that the difference solution converges to the exact one with truncation error $O(\tau^2 +h_1^2 + h_2^2)$  
(with the second order in time and space).  

\section{Difference schemes for the hyperbolic heat conduction governed by the system of equations}

For the approximate solution of the Cauchy problem for system (\ref{3.5}), (\ref{3.6}) 
we use the simplest schemes with weights  
\begin{equation}\label{5.1}
   q_{\alpha}^{\sigma(n)} +
   \nu \frac{q^{n+1}_{\alpha} - q^{n}_{\alpha} }{\tau} + 
   k A_{\alpha} u^{\sigma(n)} = 
   0, \quad  \alpha =1,2,
\end{equation}
\begin{equation}\label{5.2}
   c \frac{u^{n+1} - u^{n} }{\tau} - 
   \sum_{\alpha =1}^{2} A^*_{\alpha} q_{\alpha}^{\sigma(n)} = 
   f^{n+1/2}, \quad  n = 0,1, \ldots, N-1 ,
\end{equation}
where $\sigma$ is a numerical parameter (weight), which is usually $0 \le \sigma \le 1$. 
We have used the notation  
\[
  u^{\sigma(n)} = 
  \sigma u^{n+1} + (1-\sigma) u^{n},
  \quad q_{\alpha}^{\sigma(n)} = 
  \sigma q^{n+1}_{\alpha} + (1-\sigma) q^{n}_{\alpha},
  \quad \alpha =1,2 .
\]
For simplicity, we restrict ourselves to the same weight for all equations of system  (\ref{5.1}), (\ref{5.2}).
Taking into account (\ref{2.13}) we will supply (\ref{5.1}), (\ref{5.2}) with  the initial conditions  
\begin{equation}\label{5.3}
  u^0 = v_0, 
  \quad q_{\alpha}^0 = g_0^{(\alpha)},
  \quad \alpha =1,2 .
\end{equation}
We give the simplest estimates of stability for operator-difference scheme (\ref{5.1})--(\ref{5.3}).
Estimate (\ref{2.17}) is used to guide us.  

\begin{theorem} 
\label{t-2} 
Difference scheme (\ref{5.1})--(\ref{5.3}) is unconditionally stable at
$\sigma \geq 0.5$ and for the numerical solution the following estimate  holds
\begin{equation}\label{5.4}
  G^{n+1} \leq \exp\left (\frac{4 \tau }{T} \right ) G^{n} +
  \tau  \, T \, \exp\left (\frac{2 \sigma -1 }{T} \tau \right )
  (c^{-1} f^{n+1/2},f^{n+1/2}) ,
\end{equation}
where  
\begin{equation}\label{5.5}
  G^{n} = 
  (c u^{n}, u^{n}) +
  \nu \sum_{\alpha =1}^{2}  (k^{-1} q^{n}_{\alpha},q^{n}_{\alpha})_{\alpha} .
\end{equation}
\end{theorem} 
 
\begin{proof}
Scalarly multiply in $H$  equation (\ref{5.2})  by  $2 \tau \, u^{\sigma(n)}$,
and each separate equation (\ref{5.1}) scalarly multiply  in $H_{\alpha}$  by $2 \tau \, k^{-1} q_{\alpha}^{\sigma(n)}, \ \alpha =1,2$ 
and sum them.
Taking into account that  
\[
  2 \tau \, u^{\sigma(n)} = 
  \tau (u^{n+1} + u^{n}) +
  (2\sigma -1) \tau^2 \frac{u^{n+1} - u^{n}}{\tau} ,
\]
we obtain  
\begin{equation}\label{5.6}
  (c u^{n+1}, u^{n+1}) - (c u^{n}, u^{n}) +
  \nu \sum_{\alpha =1}^{2}  (k^{-1} q^{n+1}_{\alpha},q^{n+1}_{\alpha})_{\alpha} -
  \nu \sum_{\alpha =1}^{2}  (k^{-1} q^{n}_{\alpha},q^{n}_{\alpha})_{\alpha} 
\end{equation}
\[
  + (2\sigma -1) \tau^2 \left (c \frac{u^{n+1} - u^{n}}{\tau},
  \frac{u^{n+1} - u^{n}}{\tau} \right )
  + (2\sigma -1) \tau^2 \nu \sum_{\alpha =1}^{2} 
  \left (k^{-1} \frac{q^{n+1}_{\alpha} - q^{n}_{\alpha} }{\tau},
  \frac{q^{n+1}_{\alpha} - q^{n}_{\alpha} }{\tau} \right )
  \]
\[
  + 2 \tau \sum_{\alpha =1}^{2}  
  (k^{-1} q_{\alpha}^{\sigma(n)},q_{\alpha}^{\sigma(n)})_{\alpha} 
  = 2 \tau (f^{n+1/2}, u^{\sigma(n)}) .
\]
For terms in the right hand side of (\ref{5.6}) we have  
\[
  2 \tau (f^{n+1/2}, u^{\sigma(n)}) =
  (2\sigma -1) \tau^2 \left(f^{n+1/2}, 
  \frac{u^{n+1} - u^{n}}{\tau} \right) +
  \tau (f^{n+1/2}, u^{n+1} + u^{n} ) .
\]
We restrict ourselves to schemes with $\sigma \geq 0.5$ and use the estimates  
\[
  (2\sigma -1) \tau^2   \left (f^{n+1/2}, 
  \frac{u^{n+1} - u^{n}}{\tau} \right ) 
\]
\[
  \leq   
  (2\sigma -1) \tau^2   \left ( c \frac{u^{n+1} - u^{n}}{\tau},
  \frac{u^{n+1} - u^{n}}{\tau}\right ) +
  \frac{(2\sigma -1)}{4} \tau^2 (c^{-1} f^{n+1/2},f^{n+1/2}) ,
\]
\[
  \tau (f^{n+1/2}, u^{n+1} + u^{n} ) \leq 
  \frac{\tau}{2 T} (c (u^{n+1} + u^{n}),(u^{n+1} + u^{n}) ) +
  \frac{\tau \, T}{2} (c^{-1} f^{n+1/2},f^{n+1/2}) ,
\]
\[
  (c (u^{n+1} + u^{n}),(u^{n+1} + u^{n}) ) \leq 
  2 (c u^{n+1}, u^{n+1}) + 2 (c u^{n}, u^{n}) .
\]
Substitution in (\ref{5.6}) gives  
\begin{equation}\label{5.7}
  \left (1 - \frac{\tau}{T} \right ) (c u^{n+1}, u^{n+1}) +
  \nu \sum_{\alpha =1}^{2}  (k^{-1} q^{n+1}_{\alpha},q^{n+1}_{\alpha})_{\alpha} 
\end{equation}
\[
  \leq 
  \left (1 + \frac{\tau}{T} \right )
  (c u^{n}, u^{n}) +
  \nu \sum_{\alpha =1}^{2}  (k^{-1} q^{n}_{\alpha},q^{n}_{\alpha})_{\alpha} 
  + \frac{\tau \, T}{2} \left (1 + \frac{2 \sigma -1}{2T}\tau \right ) 
  (c^{-1} f^{n+1/2},f^{n+1/2}) .
\]
Without loss of generality, we assume that $2 \tau \leq T$ and therefore  
\[
  \left (1 + \frac{\tau}{T} \right )
  \left (1 - \frac{\tau}{T} \right )^{-1} \leq \exp\left (\frac{4 \tau }{T} \right ) .  
\]
With this in mind, from (\ref{5.7}) we obtain timelevel-wise stability estimate (\ref{5.4}), (\ref{5.5}).
\qquad\end{proof} 

A priori estimate (\ref{5.4}) is nothing but the grid analog of estimate (\ref{2.17})  and provides 
unconditional stability of the difference scheme with weights (\ref{5.1}), (\ref{5.2})
under natural conditions $\sigma \geq 0.5$. 
Considering the corresponding problem for the error \cite{samarskii2001theory,samarskii2002difference}, 
we prove the convergence of the solution of operator-difference problem (\ref{5.1})--(\ref{5.3})  
to the solution of differential-difference problem (\ref{2.1}), (\ref{2.3}), (\ref{2.13})
at $\sigma \geq 0.5$ with order $\mathcal{O}((2 \sigma -1)\tau + \tau^2)$. 
If $\sigma = 0.5$, we have the second order of convergence with respect to  $\tau$. 

The computational implementation of scheme (\ref{5.1}), (\ref{5.2})  requires 
to solve the following grid problem at new time level $n+1$: 
\begin{equation}\label{5.8}
   \sigma \tau q_{\alpha}^{n+1} +
   \nu q^{n+1}_{\alpha}  + 
   \sigma  \tau k A_{\alpha} u^{n+1} = 
   \chi^n_{\alpha}, \quad  \alpha =1,2,  
\end{equation}
\begin{equation}\label{5.9}
   c u^{n+1} - 
   \sigma  \tau \sum_{\alpha =1}^{2} A^*_{\alpha} q_{\alpha}^{n+1} = 
   \phi^{n} 
\end{equation}
for given $\chi^n_{\alpha}, \ \alpha =1,2$ and $\phi^{n}$.
Substituting $q^{n+1}_{\alpha}$  from equations (\ref{5.8})  in equation (\ref{5.8}), we obtain  
\begin{equation}\label{5.10}
   (\nu +  \sigma  \tau) c u^{n+1} + \sigma^2 \, \tau^2 
   \sum_{\alpha =1}^{2} A^*_{\alpha} A_{\alpha} u^{n+1} = 
   (\nu +  \sigma  \tau)\phi^{n} + \sigma \, \tau
   \sum_{\alpha =1}^{2} A^*_{\alpha} \chi_{\alpha}^{n} .
\end{equation}
Other components of the approximate solution are evaluated after solving grid problem 
(\ref{5.10}) via the explicit formulas of equations (\ref{5.8}).

To preserve the second order of approximation, different grids in time are often employed
for the individual components of the solution.  
The following scheme for system  (\ref{3.5}), (\ref{3.6}) provides an example
\begin{equation}\label{5.11}
   \frac{q^{n+1/2}_{\alpha} + q^{n-1/2}_{\alpha} }{2} + 
   \nu \frac{q^{n+1/2}_{\alpha} - q^{n-1/2}_{\alpha} }{\tau} + 
   k A_{\alpha} u^{n} = 
   0, \quad  \alpha =1,2,
\end{equation}
\begin{equation}\label{5.12}
   c \frac{u^{n+1} - u^{n} }{\tau} - 
   \sum_{\alpha =1}^{2} A^*_{\alpha} q_{\alpha}^{n+1/2} = 
   f^{n+1/2}, \quad  n = 0,1, \ldots, N-1 .
\end{equation}
Such explicit schemes are widely used in computational practice.  
A detailed discussion of such schemes in application to problems of electrodynamics is presented, for example, in 
\cite{taflove2000computational},  with references to works of other authors.  
The main drawback of such schemes is connected with restrictions on the time step (conditional stability).  

Considering equation (\ ref (5.12)) at two time levels, we obtain the following equations  
\[
   c \frac{u^{n+1} - u^{n-1} }{2 \tau} - 
   \sum_{\alpha =1}^{2} A^*_{\alpha} 
   \frac{q^{n+1/2}_{\alpha} + q^{n-1/2}_{\alpha} }{2} = 
   \frac{f^{n+1/2} + f^{n-1/2}}{2} ,
\]
\[
   c \frac{u^{n+1} - 2 u^{n} + u^{n-1} }{\tau^2} - 
   \sum_{\alpha =1}^{2} A^*_{\alpha} 
   \frac{q^{n+1/2}_{\alpha} - q^{n-1/2}_{\alpha} }{\tau } =
   \frac{f^{n+1/2} - f^{n-1/2}}{\tau} .
\]
Taking into account equation (\ref{5.11}), we derive  
\begin{equation}\label{5.13}
  \nu c \frac{u^{n+1} - 2u^{n} + u^{n-1}}{\tau^2} + 
  c \frac{u^{n+1} - u^{n-1}}{2\tau} + 
  \sum_{\alpha =1}^{2} A^*_{\alpha} k A_{\alpha} u^{n}
\end{equation}
\[
   =  \frac{f^{n+1/2} + f^{n-1/2}}{2} +
   \nu \frac{f^{n+1/2} - f^{n-1/2}}{\tau} .
\]
Thus we have the explicit approximation of hyperbolic equation (\ref{4.2})  with  $\sigma = 0$.

Unconditionally stable (at $\sigma \geq 0.25$) scheme with weights  
\[
  \nu c \frac{u^{n+1} - 2u^{n} + u^{n-1}}{\tau^2} + 
  c \frac{u^{n+1} - u^{n-1}}{2\tau} 
  + \sum_{\alpha =1}^{2} A^*_{\alpha} k A_{\alpha} 
  (\sigma u^{n+1} + (1- 2\sigma) u^{n} + \sigma u^{n-1})
\]
\[
   =  \frac{f^{n+1/2} + f^{n-1/2}}{2} +
   \nu \frac{f^{n+1/2} - f^{n-1/2}}{\tau} 
\]
is equivalent to the following scheme  for system (\ref{3.5}), (\ref{3.6}), 
if in scheme (\ref{5.11}), (\ref{5.12})  instead of (\ref{5.11}) we use 
\begin{equation}\label{5.14}
   \frac{q^{n+1/2}_{\alpha} + q^{n-1/2}_{\alpha} }{2} + 
   \nu \frac{q^{n+1/2}_{\alpha} - q^{n-1/2}_{\alpha} }{\tau} 
\end{equation}
\[
   + k A_{\alpha} (\sigma u^{n+1} + (1- 2\sigma) u^{n} + \sigma u^{n-1}) = 
   0, \quad  \alpha =1,2 .
\]

Scheme (\ref{5.12}), (\ref{5.14}) is not very convenient for the practical usage.  
Its main drawback results from the explicit coupling of equations for the temperature and heat fluxes.  
We must perform some preliminary work in order to obtain acceptable grid problems for evaluating 
the individual components of the solution at the new time level.
  
Starting from explicit scheme (\ref{5.11}), (\ref{5.12}), 
we can construct unconditionally stable implicit schemes.  
We can do it in the most simple way using the Samarskii principle 
of regularization for difference schemes \cite{samarskii2001theory,samarskii2002difference}, 
which is based on increasing the stability of a scheme via the perturbation of its operators.  
Stability of scheme  (\ref{5.11}), (\ref{5.12})  can be achieved in different ways.  
The most interesting possibility is connected with the multiplicative \cite{SamVabSReg,samarskii1999additive}
perturbation (increasing) of the time derivative operator or perturbation (decreasing) of the spatial variables 
operator for the individual equations of the system.  

Consider the perturbation of equation (\ref{5.12}) in detail.
The implicit scheme can be written as  
\begin{equation}\label{5.15}
   c^{1/2} \left ( E + \sigma \tau^2 
   \sum_{\alpha =1}^{2} A^*_{\alpha} A_{\alpha} \right ) c^{1/2} \,
   \frac{u^{n+1} - u^{n} }{\tau} - 
   \sum_{\alpha =1}^{2} A^*_{\alpha} q_{\alpha}^{n+1/2} = 
   f^{n+1/2} .
\end{equation}
The perturbation has  the order of $\mathcal{O}(\tau^2)$ and therefore 
regularized scheme  (\ref{5.11}), (\ref{5.15})  remains in the class of schemes with the second order approximation.  
The operator at the time derivative is selfadjoint and positive definite.  

\begin{theorem} 
\label{t-3} 
Difference scheme (\ref{5.11}), (\ref{5.15})
 is unconditionally stable at $\nu c_0 \sigma \geq 0.25$ and for the difference solution 
we have estimate (\ref{4.4}) where
\begin{equation}\label{5.16}
  S^n = \left (\left (
  \nu c^{1/2}(E + \sigma \tau^2 D)c^{1/2} - 
  \frac{\tau^2}{4}  D \right )\frac{u^n - u^{n-1}}{\tau }, 
  \frac{u^n - u^{n-1}}{\tau } \right ) 
\end{equation}
\[
  + \left (D \frac{u^n + u^{n-1}}{2}, \frac{u^n + u^{n-1}}{2} \right ) ,
\]
\[
  \varphi^n = \frac{f^{n+1/2} + f^{n-1/2}}{2} +
   \nu \frac{f^{n+1/2} - f^{n-1/2}}{\tau} . 
\]
\end{theorem} 

\begin{proof}
From (\ref{5.11}), (\ref{5.15})  in the standard way we obtain the following 
three-level scheme  
\begin{equation}\label{5.17}
  \nu c^{1/2}(E + \sigma \tau^2 D)c^{1/2} 
  \frac{u^{n+1} - 2u^{n} + u^{n-1}}{\tau^2} 
\end{equation}
\[
  + c^{1/2}(E + \sigma \tau^2 D) c^{1/2} \frac{u^{n+1} - u^{n-1}}{2\tau} 
  + D u^{n}  = \varphi^n , 
  \quad n = 1,2, \ldots, N-1 . 
\]
Further investigation is conducted similarly to the proof of Theorem~\ref{t-1}. 
Under the above restrictions on the weight we have  
\[
  \nu c^{1/2}(E + \sigma \tau^2 D)c^{1/2} - 
  \frac{\tau^2}{4}  D > 0
\]
and we associate the squared norm of the difference solution with $S^n$.
\qquad\end{proof} 

The numerical implementation of scheme  (\ref{5.11}), (\ref{5.15}) 
is based on inversion of the same grid elliptic operator $E + \sigma \tau^2 D$, 
whereas the schemes with weights for the hyperbolic heat conduction equation 
(\ref{4.2}) requires to invert  $c(\nu + \tau) E + \sigma \tau^2 D$.
You can also obtain a grid analog of (\ref{2.16}), (\ref{2.17}) for difference scheme (\ref{5.11}), (\ref{5.15}). 
However, it seems difficult to proof same analog  of  Theorem~\ref{t-2} in this.

\section{Splitting scheme for the hyperbolic heat conduction equation}

The above considered unconditionally stable operator-difference schemes --- (\ref{4.2}) 
for the hyperbolic heat equation and (\ref{5.11}), (\ref{5.15})
for the system of hyperbolic heat conduction, respectively, --- are not very convenient in the numerical implementation.  
We construct the additive schemes for problem (\ref{2.5}), (\ref{2.7}), 
where the transition to a new time level will be connected
with the solution of more simple problems related to the inversion  of 
individual operators $A^*_{\alpha} A_{\alpha},  \ \alpha =1,2$, 
rather than their sum (operator $D$ in (\ref{4.2})).
Taking into account the nature of operators $A^*_{\alpha}, A_{\alpha},  \ \alpha =1,2$, 
we are talking about locally one-dimensional schemes \cite{samarskii2001theory}.

We will focus on using regularized additive schemes of full approximation  \cite{samarskii1998regularized,samarskii1992regularized}.
The principle of regularization of difference schemes is used traditionally widely 
\cite{samarskii2001theory}  to construct stable difference schemes for the
numerical solution of problems governed by partial differential equations.  
Due to small perturbations of the problem operators we can control the growth of the norm 
for the solution at the transition from one time level to another.  

The construction of unconditionally stable difference schemes  via 
the principle of regularization is implemented as follows.  
For the initial problem there is constructed some simple difference scheme 
(producing difference scheme) which does not meet the necessary properties, 
 ie the  scheme is conditionally stable or even absolutely unstable. 
Then the quality of the difference scheme (its stability) is improved via 
perturbations of the difference scheme operators.  

It is natural to consider as the producing schemes the following explicit scheme  
\begin{equation}\label{6.1}
  \nu c \frac{u^{n+1} - 2u^{n} + u^{n-1}}{\tau^2} + 
  c \frac{u^{n+1} - u^{n-1}}{2\tau} + 
  D u^{n}  = \varphi^n , 
\end{equation}
which is complemented by initial conditions  (\ref{4.3}).
The stability of this scheme (see (\ref{4.3}) at $\sigma = 0$)  will be provided 
if the following inequality holds  
\begin{equation}\label{6.2}
  R = \nu c E - \frac{\tau^2}{4} D  > 0. 
\end{equation}
In this case we have  estimate (\ref{4.4}), in which  
\begin{equation}\label{6.3}
  S^n = \left (R \frac{u^n - u^{n-1}}{\tau }, 
  \frac{u^n - u^{n-1}}{\tau } \right ) 
  + \left (D \frac{u^n + u^{n-1}}{2}, \frac{u^n + u^{n-1}}{2} \right ) .
\end{equation}
Taking into account  (\ref{3.22}),  from (\ref{6.2}) we obtain the condition 
for stability of explicit scheme  (\ref{6.1})
\[
  \tau^2 \leq \frac{4 \nu c_0}{k_1 (\Delta_1 + \Delta_2)}
  = \mathcal{O}(h_1^{-2} + h_2^{-2}) .
\]
To increase the stability limit (increase operator $R$), 
we can employ the perturbation of both the first term in $R$ ($\nu c E $)
and second one ($D$). 

In the case of perturbing the operator for the second time derivative  we construct the regularized 
scheme by analogy with (\ref{5.11}), (\ref{5.15}):
\begin{equation}\label{6.4}
  \nu c^{1/2} Q  c^{1/2} \frac{u^{n+1} - 2u^{n} + u^{n-1}}{\tau^2} + 
  c^{1/2} Q  c^{1/2} \frac{u^{n+1} - u^{n-1}}{2\tau} + 
  D u^{n}  = \varphi^n , 
\end{equation}
where operator $Q = Q^* = E + \mathcal{O} (\tau^2)$.
In the construction of additive schemes, we need to take into account 
the structure of the grid operator at new time level.  
Assume that 
\begin{equation}\label{6.5}
  Q = \left (E + \frac{\sigma}{2} \tau^2 A_2^* A_2 \right )
  \left (E + \sigma \tau^2 A_1^* A_1 \right )
  \left (E + \frac{\sigma}{2} \tau^2 A_2^* A_2 \right ) ,
\end{equation}
so that  $Q > E + \sigma \tau^2 D$.
Direct calculations verify that for scheme (\ref{6.4}), (\ref{6.5}) instead of (\ref{6.2}) we have  
\begin{equation}\label{6.6}
  R = \nu c^{1/2} Q  c^{1/2} - \frac{\tau^2}{4} D  
\end{equation}
and  $R > 0$ at $\nu c_0 \sigma \geq 0.25$. 

\begin{theorem} 
\label{t-4} 
Additive-difference scheme (\ref{6.4}), (\ref{6.5})
is unconditionally stable at $\nu c_0 \sigma \geq 0.25$
and estimate  (\ref{4.4}), \ref{6.3}), (\ref{6.6})
is valid for the difference solution.
\end{theorem} 

The second possibility of constructing unconditionally stable additive 
operator-difference schemes is connected with the perturbation of operator $D$  in explicit scheme (\ref{6.1}).  
Instead of operator $D$, which is defined according to (\ref{3.18}),  we use  
\begin{equation}\label{6.7}
  C = \sum_{\alpha =1}^{2} C_{\alpha},
  \quad C_{\alpha} = A^*_{\alpha} 
    \left (k^{-1} E + \sigma \tau^2 A_{\alpha} A^*_{\alpha} \right )^{-1}
  A_{\alpha},
  \quad \alpha = 1,2 .  
\end{equation}
For these difference operators  
\[
  C_{\alpha} = C^*_{\alpha} < \frac{1}{\sigma \tau^2} E, 
  \quad \alpha = 1,2 .  
\]
Because of this, for  
\begin{equation}\label{6.8}
  R = \nu c E - \frac{\tau^2}{4} C  
\end{equation}
we have $R > 0$ at $\nu c_0 \sigma \geq 0.5$.

\begin{theorem} 
\label{t-5} 
Additive difference schemes  
\begin{equation}\label{6.9}
  \nu c \frac{u^{n+1} - 2u^{n} + u^{n-1}}{\tau^2} + 
  c \frac{u^{n+1} - u^{n-1}}{2\tau} + 
  C u^{n}  = \varphi^n ,
\end{equation}
where operator $C$ is defined according to (\ref{6.7}),  
is unconditionally stable at $\nu c_0 \sigma \geq 0.5$.  Estimate  (\ref{4.4})  is true
for the difference solution  with
\begin{equation}\label{6.10}
   S^n = \left (R \frac{u^n - u^{n-1}}{\tau }, 
  \frac{u^n - u^{n-1}}{\tau } \right ) 
  + \left (C \frac{u^n + u^{n-1}}{2}, \frac{u^n + u^{n-1}}{2} \right ) 
\end{equation}
and $R$  corresponding to  (\ref{6.8}).
\end{theorem} 

The main computational cost in the constructed splitting schemes results from the inversion
of one-dimensional grid operators $k^{-1} E + \sigma \tau^2 A_{\alpha} A^*_{\alpha}, \ \alpha = 1,2$.
The potential advantage of  additive scheme (\ref{6.7}), (\ref{6.9})
in compare with scheme (\ref{6.4}), (\ref{6.5}) is connected primarily with lower computational 
cost during the transition to a new time level.
This advantage is more impressive  for the three-dimensional problems (splitting in three directions).  

\section{Additive schemes for the hyperbolic heat conduction governed by the  system of equations}

In the construction of splitting schemes for the system of equations
governing the hyperbolic heat conduction,  the theory and practice of using additive 
schemes for first order evolutionary equations will be employed.  
We treat the system of equations (\ref{3.15}), (\ref{3.16})  
as a single evolutionary equation for the vector 
$\mathbf{u} \equiv  [u_1, u_2, u _3]^T = [q_1, q_2, u ]^T$:
\begin{equation}\label{7.1}
  \mathbf{B}\frac{d \mathbf{u}}{d t} + \mathbf{A} \mathbf{u} = \mathbf{f}(t),
\end{equation}
where  $\mathbf{f} = [0, 0, f]^T$.
For the elements of operator matrices $\mathbf{B}$ and  $\mathbf{A}$ we have the representation  
\begin{equation}\label{7.2}
  \mathbf{B} = 
  \left[ 
    \begin{array}{ccc}
      \nu k^{-1} & 0 & 0 \\
      0 & \nu k^{-1} &  0 \\
	  0 & 0& c
	\end{array}
  \right],
  \qquad 
  \mathbf{A} = 
  \left[ 
    \begin{array}{ccc}
      k^{-1} & 0 & A_1 \\
      0 & k^{-1} &  A_2 \\
	  - A_1^* & - A_2^* & 0
	\end{array}
  \right] .
\end{equation}
For the direct sum of spaces $\mathbf{H} = H_1 \oplus  H_2 \oplus H$, we set  
\[
  (\mathbf{u}, \mathbf{v}) = \sum_{\alpha =1}^{p} (u_{\alpha},v_{\alpha})_{\alpha} ,
  \quad \|\mathbf{u} \|^2 = \sum_{\alpha =1}^{p} \|u_{\alpha}\|^2_{\alpha} .
\]
In this case, $\mathbf{A} > 0$  in $\mathbf{H}$ and estimate (\ref{2.16}), (\ref{2.17}) 
can be rewritten as  
\begin{equation}\label{7.3}
 \left (\mathbf{B} \mathbf{u}(t),\mathbf{u}(t) \right ) \leq 
 \exp(t)\left (\mathbf{B} \mathbf{u}(0),\mathbf{u}(0) \right ) + 
 \int\limits_0^t \exp(t-\theta)
  \left (\mathbf{B}^{-1} \mathbf{f}(\theta), \mathbf{f}(\theta)
  \right )^2 d \theta .
\end{equation}

To construct locally one-dimensional schemes for the Cauchy problem for (\ref{7.1}), (\ref{7.2}), 
we use the additive representation of operator $\mathbf{A}$ in the form 
\begin{equation}\label{7.4}
  \mathbf{A} = \sum_{\alpha =1}^{p} \mathbf{A}^{(\alpha)} .
\end{equation}
The first variant of  decomposition (\ref{7.4})  corresponds to the selection $p=3$  and  
\begin{equation}\label{7.5}
  \mathbf{A}^{(1)} = 
  \left[ 
    \begin{array}{ccc}
      0 & 0 & A_1 \\
      0 & 0 &  0 \\
	  - A_1^* & 0 & 0
	\end{array}
  \right] ,
  \quad 
  \mathbf{A}^{(2)} = 
  \left[ 
    \begin{array}{ccc}
      0 & 0 & 0 \\
      0 & 0 &  A_2 \\
	  0 & - A_2^* & 0
	\end{array}
  \right] ,
\end{equation}
\[
  \mathbf{A}^{(3)} = 
  \left[ 
    \begin{array}{ccc}
      k^{-1} & 0 & 0 \\
      0 & k^{-1} &  0 \\
	  0 & 0 & 0
	\end{array}
  \right] .
\]
Thus, we separate the individual terms with operators  
$A_{\alpha}, \ A^*_{\alpha}, \ \alpha =1,2$.
The main properties of these operators are connected with their non-negativity 
\[
  \mathbf{A}^{(\alpha)} = - (\mathbf{A}^{(\alpha )})^*,
  \quad \alpha = 1,2,
  \quad \mathbf{A}^{(3)} = (\mathbf{A}^{(3)})^* \geq 0 .
\]
We can consider the splitting of (\ref{7.4}) so that  $p=2$ where  
\begin{equation}\label{7.6}
  \mathbf{A}^{(1)} = 
  \left[ 
    \begin{array}{ccc}
      0.5 k^{-1} & 0 & A_1 \\
      0 & 0.5 k^{-1} &  0 \\
	  - A_1^* & 0 & 0
	\end{array}
  \right] ,
  \quad 
  \mathbf{A}^{(2)} = 
  \left[ 
    \begin{array}{ccc}
      0.5 k^{-1} & 0 & 0 \\
      0 & 0.5 k^{-1} &  A_2 \\
	  0 & - A_2^* & 0
	\end{array}
  \right] .
\end{equation}
In this case $\mathbf{A}^{(\alpha)} \geq 0, \ \alpha = 1,2$
in  space $\mathbf{H}$.

This non-negative property of operators $\mathbf{A}^{(\alpha)}, \ \alpha = 1,2,...,p$ 
in splitting  (\ref{7.4}) allows to use for the approximate solution of the Cauchy 
problem for equation (\ref{7.1}), (\ref{7.2}) different classes of unconditionally stable 
additive operator-difference schemes   \cite{marchuk1990splitting,samarskii1999additive}. 
With regard to our problem of the hyperbolic heat conduction, we can employ the schemes of second-order approximation in time.

For the general case (in (\ref{7.4}) $p>2$) 
the standard additive schemes are based on the concept 
of summarized  approximation.  
To construct the schemes of second order, we arrange computations via the algorithm
\[
   \frac12 \mathbf{A}^{(1)} \rightarrow  \frac12 \mathbf{A}^{(2)} 
   \rightarrow  \cdots \rightarrow  \frac12 \mathbf{A}^{(p)} 
   \rightarrow  \frac12 \mathbf{A}^{(p)}
   \rightarrow  \frac12 \mathbf{A}^{(p-1)} \rightarrow  \cdots \rightarrow  
   \frac12 \mathbf{A}^{(1)} .
\]
The corresponding additive scheme of component-wise splitting seems like this:  
\begin{equation}\label{7.7}
   \mathbf{B}\frac{\mathbf{u}^{n+\alpha /(2p)} - \mathbf{u}^{n+(\alpha-1) /(2p)} }{\tau} + 
   \frac{1}{2} \tilde{\mathbf{A}}^{(\alpha )} 
   (\mathbf{u}^{n+\alpha /(2p)} + \mathbf{u}^{n+(\alpha-1) /(2p)})
\end{equation}
\[
    = \mathbf{f}_{\alpha }^{n+1/2}, \quad  \alpha = 1,2,\dots,2p,
\]
where
$\tilde{\mathbf{A}}^{(\alpha )} = \mathbf{A}^{(\alpha )}, \ \alpha = 1,2,...,p$,
$\tilde{\mathbf{A}}^{(\alpha )} = \mathbf{A}^{(2p+1-\alpha )}, \ \alpha = p+1,p+2,...,2p$ è
\[
  \mathbf{f}^{n+1/2} = \sum_{\alpha =1}^{2p-2}
  \mathbf{f}_{\alpha }^{n+1/2} .
\]

The proof of stability and convergence is conducted in the standard way, the technical details can be found, for example, in  \cite{samarskii1999additive}. 
Scalarly multiplying the equations of scheme (\ref{7.7}) by $\mathbf{y}^{n+\alpha /(2p)} + \mathbf{y}^{n+(\alpha-1) /(2p)}$,
we obtain  the corresponding analog of a priori estimate (\ref{7.3}). 

\begin{theorem}
\label{t-6}
Additive operator-difference scheme of summarized  approximation (\ref{7.7}) is
 unconditionally stable and approximates the system of equations (\ref{7.1})  with the second order relative to $\tau$.
\end{theorem}

The computational implementation of the considered additive schemes is much 
simpler than for schemes (\ref{5.1}), (\ref{5.2}) or (\ref{5.11}), (\ref{5.15}).
To explain this fact, we consider, for example, the first step in scheme (\ref{7.7}) 
with splitting (\ref{7.5}),  where  
\[
   \left ( \mathbf{B} + \frac{\tau }{2} \mathbf{A}^{(1)}\right ) 
   \mathbf{u}^{n+1/6} 
   = \mathbf{r}^{n+1/2} 
\]
for  given right hand side $\mathbf{r}^{n+1/2}$.
In the coordinate-wise form of this equation we have the system of equations  
\[
  \nu u_1^{n+1/6} + \frac{\tau }{2} k A_1  y_3^{n+1/6} =  k r_1^{n+1/2},
  \quad \nu u_2^{n+1/6} =  k r_2^{n+1/2},
\]
\[
  c u_3^{n+1/6} - \frac{\tau }{2} A_1^*  u_1^{n+1/6} =  r_3^{n+1/2}.
\]
Substituting $u_1^{n+1/6}$ from  the first equation into the third one, we obtain  
\[
  \nu c u_3^{n+1/6} + \frac{\tau^2 }{4} A_1^* k A_1 u_3^{n+1/6} =  
  \nu r_3^{n+1/2} + \frac{\nu \tau }{2} A_1^* k r_1^{n+1/2}.
\]
Thus we must solve the one-dimensional grid problems with a single operator, 
which are connected with operators  $A_1, A^*_1$.
We have a similar realization for splitting (\ref{7.6}).

Among  shortcomings of the constructed locally one-dimensional schemes (\ref{7.7})  
it should be noted the lack of transparency (each individual equation does not 
approximate the differential problem) as well as the relative difficulty of obtaining
and studying schemes of increased  approximation order. 
It is possible to construct for system (\ref{3.15}), (\ref{3.16}) another splitting schemes
which belong to the class of regularized additive operator-difference schemes  \cite{samarskii1999additive}. 

Regularized additive schemes can be constructed on the basis of  scheme (\ref{5.11}), (\ref{5.15}). 
Instead of (\ref{5.15})  we use the difference equation 
 \begin{equation}\label{7.8}
   c^{1/2} Q c^{1/2} \,
   \frac{u^{n+1} - u^{n} }{\tau} - 
   \sum_{\alpha =1}^{2} A^*_{\alpha} q_{\alpha}^{n+1/2} = 
   f^{n+1/2} ,
\end{equation}
where the factorized operator $Q$ is defined according to (\ref{6.5}).
Similar to Theorem~\ref{t-3},  we can prove the following statement involving estimate $Q > E + \sigma \tau^2 D$. 

\begin{theorem} 
\label{t-7} 
Additive-difference scheme (\ref{5.11}), (\ref{6.5}), (\ref{7.8}) is unconditionally stable at $\nu c_0 \sigma \geq 0.25$,
and  estimate (\ref{4.4}) holds for the difference solution, where 
 \[
  S^n = \left (\left (
  \nu c^{1/2}Q c^{1/2} - 
  \frac{\tau^2}{4}  D \right )\frac{u^n - u^{n-1}}{\tau }, 
  \frac{u^n - u^{n-1}}{\tau } \right ) 
\]
\[
  + \left (D \frac{u^n + u^{n-1}}{2}, \frac{u^n + u^{n-1}}{2} \right ) .
\]
\end{theorem} 

Additive scheme (\ref{5.11}), (\ref{6.5}), (\ref{7.8}) 
is based on the perturbation of  the operator at the time derivative in the last equation of system (\ref{5.11}), (\ref{5.12})
(difference derivative of the temperature).   
It is interesting to consider the schemes with the perturbation of difference derivatives in time for heat fluxes.   
Instead of (\ref{5.11}) we use the difference equations  
\begin{equation}\label{7.9}
   (k^{-1} E + \sigma \tau^2  A_{\alpha} A^*_{\alpha} )
   \frac{q^{n+1/2}_{\alpha} + q^{n-1/2}_{\alpha} }{2}  
\end{equation}
\[
   + \nu (k^{-1} + \sigma \tau^2 A_{\alpha} A^*_{\alpha} )
   \frac{q^{n+1/2}_{\alpha} - q^{n-1/2}_{\alpha} }{\tau}  
   + A_{\alpha} u^{n} = 
   0, \quad  \alpha =1,2 .
\]
Difference equations (\ref{7.9}) can be written in the form 
\[
   \frac{q^{n+1/2}_{\alpha} + q^{n-1/2}_{\alpha} }{2}  
   + \nu \frac{q^{n+1/2}_{\alpha} - q^{n-1/2}_{\alpha} }{\tau}  
\]
\[
   + (k^{-1} + \sigma \tau^2 A_{\alpha} A^*_{\alpha} )^{-1}
   A_{\alpha} u^{n} = 0, 
   \quad  \alpha =1,2 
\]
with treating them as the multiplicative perturbation of operators  
$A_{\alpha}, \ \alpha =1,2$ â (\ref{5.11}).

It is easy to see by direct calculations that  additive scheme (\ref{5.12})(\ref{7.9}) 
corresponds to scheme (\ref{6.7})(\ref{6.9}).

\begin{theorem} 
\label{t-8} 
Additive-difference scheme (\ref{5.12}), (\ref{7.9}) is unconditionally stable at  $\nu c_0 \sigma \geq 0.5$,
and  estimate (\ref{6.10}) holds for the difference solution under the condition of
 setting operators $C$  and $R$  according to (\ref{6.7})  and (\ref{6.8}), respectively.  
\end{theorem} 

In contrast to the scheme of summarized  approximation (\ref{7.7}), operator-difference splitting scheme (\ref{5.11}), (\ref{7.9})
has a clear and transparent structure.

\end{document}